\documentclass[12pt]{article}
\usepackage{cite}
\usepackage{fullpage}
\usepackage{amsmath,amsthm,amscd,amsfonts,mathrsfs,amssymb}
\usepackage{enumerate}
\DeclareMathOperator{\Null}{null}
\newcommand{\lgd}[2]{\left(\frac{#1}{#2}\right)}
\newcommand{\Pmod}[1]{~(\rm{mod}\ {#1})}
\newtheorem{thm}{Theorem}[section]
\newtheorem{cor}[thm]{Corollary}
\newtheorem{lem}[thm]{Lemma}
\newtheorem{prop}[thm]{Proposition}

\theoremstyle{definition}
\newtheorem{defn}[thm]{Definition}
\numberwithin{equation}{section}
\DeclareMathOperator{\coker}{coker} 
 
 \DeclareMathOperator{\ch}{char}
\DeclareMathOperator{\rk}{rk} \DeclareMathOperator{\rank}{rank}
\newcommand{\set}[1]{\left\{\,#1\,\right\}}
\newcommand{\Set}[2]{\left\{\,{#1}\;\middle|\;{#2}\,\right\}}
\newcommand{\pr}[1]{\left({#1}\right)} 
\newcommand{\map}[1]{\ensuremath\xrightarrow{{#1}}}
\newcommand{\mto}{\ensuremath\mapsto}
\newcommand{\F}{\ensuremath\mathbb{F}}
\newcommand{\Q}{\ensuremath\mathbb{Q}}
\newcommand{\Z}{\ensuremath\mathbb{Z}}
\newcommand{\fa}{\ensuremath\mathfrak{a}}
\newcommand{\fb}{\ensuremath\mathfrak{b}}
\newcommand{\fp}{\ensuremath\mathfrak{p}}
\newcommand{\cA}{\ensuremath\mathcal{A}}
\newcommand{\cB}{\ensuremath\mathfrak{B}}

\newcommand{\sB}{\ensuremath\mathscr{B}}
\newcommand{\sC}{\ensuremath\mathscr{C}}

\newcommand{\sN}{\ensuremath\mathscr{N}}

\newcommand{\inv}{\ensuremath{^{-1}}}

\newcommand{\df}{\ensuremath{:=}}
\newcommand{\mat}[4]{\ensuremath
                     \bigl( \begin{smallmatrix}
                                {#1} & {#2}\\
                                {#3} & {#4}
                             \end{smallmatrix}
                     \bigr)}
\newcommand{\op}{\oplus}
\begin{document}
\title{$8$-rank of the class group and isotropy index}
\author{Qing Lu\thanks{Academy of Mathematics and Systems Science, Chinese Academy of Sciences, Beijing 100190, China; email: \texttt{qlu@amss.ac.cn}.}}
\maketitle
\begin{abstract}
  Suppose $F=\Q(\sqrt{-p_1\dotsm p_t})$ is an imaginary quadratic number field with distinct primes $p_1, \dots, p_{t}$, where $p_i\equiv 1\Pmod{4}$ ($i=1,\dots,t-1$) and $p_t\equiv 3\Pmod{4}$. We express the possible values of the $8$-rank $r_8$ of the class group of $F$ in terms of a quadratic form $Q$ over $\F_2$ which is defined by quartic symbols. In particular, we show that $r_8$ is bounded by the isotropy index of $Q$.
\end{abstract}

\section{Introduction}
Let $t$ be a fixed positive integer. Suppose $F=\Q(\sqrt{-p_1\dotsm p_t})$ is an imaginary quadratic number field with distinct primes $p_1$, \dots, $p_{t}$, where $p_i\equiv 1\Pmod{4}$ ($i=1,\dots,t-1$) and $p_t\equiv 3\Pmod{4}$. We will study the $8$-rank of the class group $\sC$ of $F$.
Recall that the $2^k$-rank of $\sC$ is defined to be $r_{2^k}=\dim_{\F_2} 2^{k-1}\sC/2^k\sC$ for $k\ge 1$.

The study of the 2-primary part of class groups or narrow class groups of quadratic number fields can be traced back to Gauss, who proved that $r_2=t-1$. In a series of papers R\'edei and Reichardt investigate $r_{2^k}$ ($k\ge 2$) and provided an algorithm for $r_4$ by the so-called \emph{R\'edei matrix}~\cite{Redei-Reichardt, Redei1936, Redei1938, Redei1939, Redei1953}. Waterhouse~\cite{MR0319944} found a method to compute $r_8$, which was further generalized to higher R\'edei matrices by Kolster~\cite{MR2296831}. However, their algorithms required determining Hilbert symbols of solutions of Diophantine equations. Yue~\cite{MR2805786} gave more explicit solutions to the $8$-rank problem for the special case $t=2$.

In this paper we study the possible values of $r_8$ for $F=\Q(\sqrt{-p_1\dotsm p_t})$ for an arbitrary $t$. We first compute the Hilbert symbols explicitly (Theorem \ref{thm:Hilbert}). Generalizing a result of Morton~\cite{MR656862}, our main result (Theorem \ref{thm:main}) describes the possible values of the $8$-rank of $\sC$ in terms of a quadratic form over $\F_2$ which is defined by quartic symbols. Consequently, the $8$-rank is bounded by the \emph{isotropy index} (Definition~\ref{defn:isotropy}) of the quadratic form. The proof of the main theorem is given in \S\ref{sec:proof}. 


This paper is motivated by Ye Tian's recent work on the Birch and Swinnerton-Dyer conjecture, in which the 2-primary parts of class groups of quadratic number fields are related to the 2-descent method of elliptic curves~\cite{MR3023667, Tian-arXiv}.
\paragraph{Acknowledgements}  The author thanks Ye Tian for introducing the problem of $8$-rank of class groups of quadratic number fields to her. She is very grateful to Weizhe Zheng for many helpful discussions and constant encouragement. She also thanks the referees for many useful comments. This work was partially supported by China Postdoctoral Science Foundation Grant 2013M541064, National Natural Science Foundation of China Grant 11371043 and the 973 Program Grant 2013CB834202. 

\section{R\'edei matrices revisited}
R\'edei matrices are tools to study the $2$-primary part of class groups of quadratic number fields.  Here we include a review for the case $F=\Q(\sqrt{-p_1\dotsm p_t})$, where $p_i\equiv 1\Pmod{4}$ ($i=1,\dots,t-1$) and $p_t\equiv 3\Pmod{4}$. Instead of the matrix form used in the literature, here we express the R\'edei matrix for the $8$-rank as a bilinear form. In the next section we will show that the quadratic form induced from this bilinear form can be computed from quartic symbols.

For more details about R\'edei matrices, one may refer to the survey by Stevenhagen~\cite{MR1345183}, which also included some results for $\ell$-primary parts where $\ell$ is an odd prime.

In the following, let $\Delta$ be the discriminant of $F$. Then $\Delta=-p_1\dotsm p_t$.
\subsection{The 2-rank of $\sC$}\label{subsec:2-rank}
Let $V=\sC[2]$ be the group of elements of order 2 in $\sC$.

For $i=1,\dots,t$, let $\fp_i$ be the prime ideal of $F$ such that $\fp_i^2=(p_i)$. Then $\fp_1, \dots, \fp_{t}$ generate $V$. The only nontrivial relation in $\sC$ among these elements is
\begin{equation}
  \fp_1\dotsm\fp_t=1.
  \label{eqn:rel_elem}
\end{equation}
We view $V$ as a vector space over $\F_2$, the vector addition being the multiplication of ideal classes. It has a basis $\set{\fp_1,\dots,\fp_{t-1}}$. Then
\[
r_2(\sC) = \dim_{\F_2} V = t-1.
\]

For later use we introduce another vector space $V'$ over $\F_2$ which consists of all positive divisors of $p_1\dotsm p_{t-1}$. The vector addition is defined to be $q_1\cdot q_2=q_1q_2/(\gcd(q_1,q_2))^2$. Then 
\begin{equation}
  V' \to V,\qquad
  p_i \mto \fp_i
  \label{map:VV}
\end{equation}
is an isomorphism of vector spaces. We will identify $V'$ with $V$.
\subsection{The $4$-rank of $\sC$}\label{subsec:$4$-rank}
We will use the quadratic characters to study the $4$-rank of $\sC$.
Note that $r_4(\sC)=\dim_{\F_2}\sC[2]\cap 2\sC$. The elements in $\sC[2]\cap 2\sC$ are exactly the elements in $\sC[2]$ which are killed by all quadratic characters of $\sC$.

For $i=1,\dots,t$, the characters $\chi_{p_i}$ defined by
\[
\chi_{p_i}(\fa)=\lgd{N\fa,\Delta}{p_i}\]
generate $(\sC/2\sC)^\vee$, the group of quadratic characters on $\sC$.
Here $\fa$ is any (fractional) ideal of $F$, and $\lgd{a,b}{p_i}$ is the Hilbert symbol.
The only nontrivial relation among these generators is:
\begin{equation}
  \chi_{p_1}\dotsm\chi_{p_t}=1.
  \label{eqn:rel_char}
\end{equation}

We view $(\sC/2\sC)^\vee$ as a vector space over $\F_2$, the vector addition being the multiplication of quadratic characters.
It has a basis $\set{\chi_{p_1},\dots,\chi_{p_{t-1}}}$ and there is an isomorphism of vector spaces over $\F_2$:
\begin{eqnarray*}
  V=\sC[2] &\map{\cong}& (\sC/2\sC)^\vee,\\
  \fp_i &\mto& \chi_{p_i}.
\end{eqnarray*}

We restrict the quotient map $\sC\to\sC/2\sC$ to $\sC[2]$ and get
\[
 f_\cA\colon V=\sC[2] \to \sC/2\sC \cong V^\vee.
\]
This induces a bilinear form on $V$:
\[
 \cA\colon V\times V \to \F_2.
\]
Under the basis $\set{\fp_1,\dots,\fp_{t-1}}$, we may write $\cA$ in the matrix form:

\begin{defn}[R\'edei matrix]
Let $\xi\colon\set{\pm 1}\to\F_2$ be the group isomorphism defined by $\xi(1)=0$ and $\xi(-1)=1$.
The R\'edei matrix (for the $4$-rank) is defined to be 
\[
M_4=\pr{\xi\lgd{p_i,\Delta}{p_j}}_{1\le i\le t-1,\,1\le j\le t-1}\quad\text{(over $\F_2$)}.
\]
\end{defn}

We need the following result to compute the Hilbert symbol:
\begin{lem}[cf.~{\cite[Chapter V, Theorem 3.6]{MR3058613}}]\label{lem:Hilbert} Let $p\ne 2$ be a prime number. For $a$, $b\in\Q_p^*$, we write 
  \[
  a = p^\alpha a',\quad b= p^\beta b',
  \]
  where $a'$ and $b'$ are units in $\Q_p$. Then 
  \[
  \lgd{a,b}{p}
  = (-1)^{\frac{p-1}{2}\alpha\beta}
  \lgd{a'}{p}^\beta \lgd{b'}{p}^\alpha.
  \]
Here $\lgd{\hphantom{p}}{}$ is the Jacobi symbol.
\end{lem}

\begin{prop}
  Under our assumption,
  $M_4$ is a symmetric matrix. In other words, $\cA$ is a symmetric bilinear form.
\end{prop}
\begin{proof}
If $p_i\ne p_j$, we have
\[
\lgd{p_i,\Delta}{p_j}
= \lgd{p_i}{p_j}
\]
by Lemma~\ref{lem:Hilbert}. 

Since $p_i\equiv 1\Pmod{4}$ for $i=1,\dots,t-1$, we see that 
\[
\lgd{p_j}{p_i}=\lgd{p_i}{p_j}(-1)^{\frac{p_i-1}{2}\frac{p_j-1}{2}}=\lgd{p_i}{p_j} \quad\text{for all $i\ne j$}.
\]
\end{proof}

Note that $\ker f_\cA = \sC[2]\cap 2\sC$. Therefore we have the following result:
\begin{prop}[R\'edei]
  The $4$-rank of $\sC$ is
  \[
  r_4=\dim_{\F_2}V - \dim_{\F_2}f_{\cA}(V) = t-1-\rank_{\F_2} M_4.
  \]
\end{prop}

\subsection{The $8$-rank of $\sC$}\label{subsec:$8$-rank}
We write $V_0=\ker f_\cA=\sC[2]\cap 2\sC$, where $\cA$ is the bilinear form defined in \S\ref{subsec:$4$-rank}.

In \S\ref{subsec:$4$-rank} we have seen that $\cA$ is a symmetric bilinear form. Therefore, there is a natural isomorphism of vector spaces over $\F_2$:
\[
\sC/(\sC[2]+2\sC) = \coker f_\cA \cong V_0^\vee.
\]

We have the following homomorphism of vector spaces over $\F_2$:
\[
f_\cB\colon V_0=\sC[2]\cap 2\sC
\map{\div 2} \sC/\sC[2]
\to \sC/(\sC[2]+2\sC)\cong V_0^\vee.
\]
The first map ``division-by-2'' in $2\sC$ means taking square roots of the ideal classes.
The second map is the natural quotient homomorphism. Since $\ker f_\cB=\sC[2]\cap 4\sC$, we have 
\[
r_8=\dim_{\F_2} V_0-\dim_{\F_2}f_\cB(V_0).
\]

We are going to describe $f_\cB$ more explicitly. 
Let us identify $\sC[2]\cap 2\sC$ and its preimage $V_0'$ in $V'$ under the isomorphism in (\ref{map:VV}), and write both of them as $V_0$.
Consider $\fa\in V_0=\ker \cA=\sC[2]\cap 2\sC$. We will take the square root of $\fa$ as follows:

\begin{lem}[cf.~{\cite[Lemma 5]{MR656862}}]\label{lem:Morton}
   Assume $\fa$ is a proper ideal of the ring of integers of $F=\Q(\sqrt{\Delta})$ such that $N\fa\mid\Delta$, and that the class of $\fa$ in $\sC$ belongs to $2\sC$. For any positive primitive solution $(x,y,z)$ of $x^2=\Delta y^2+4a z^2$ where $a=N\fa$, there exists an ideal $\fb$ for which the class of $\fb^2$ in $\sC$ coincides with the class of $\fa$ in $\sC$ and $N\fb=z$.
\end{lem}
The existence of $(x,y,z)$ is ensured by the assumption that the class of $\fa$ in $\sC$ belongs to $2\sC$.

Let $\cB\colon V_0\times V_0 \to \F_2$ be the bilinear form induced from $f_\cB$, i.e., if $\fa\in V_0$, $D\in V_0'\cong V_0$, and
$\fb$ is the square root of $\fa$ as in Lemma~\ref{lem:Morton},
then
\[
\cB(\fa,D) = \xi(\chi_D(\fb)) = \xi\pr{\lgd{N\fb,\Delta}{D}}.
\]

Summarizing the above discussions, we obtain the following two propositions:

\begin{prop}
  With the above notations,
  \[
  \cB(\fa,D)=\xi\pr{\lgd{N\fb,\Delta}{D}}=\xi\pr{\lgd{z,\Delta}{D}}.
  \]
\end{prop}

\begin{prop}
The $8$-rank
$r_8 = \dim_{\F_2}V_0 - \dim_{\F_2} f_\cB(V_0) = r_4 - \rank_{\F_2} \cB$.
\end{prop}

\section{The quadratic form $Q_\cB$}\label{sec:Q_B}
\begin{defn}\label{defn:qf} 
  Let $W$ be a vector space over a field $k$. Then a map $Q\colon W\to k$ is called a \emph{quadratic form} if there exists a bilinear form $B$ such that $Q(x)=Q_B(x)\df B(x,x)$.
\end{defn} 
\paragraph{Remark.} 
\begin{enumerate}[(1)]
\item Note that if $\ch k\ne 2$, there is always a symmetric bilinear form $B$ satisfying the condition. When $\ch k=2$, this is no longer true. 
\item By \cite[\S3, no.~4, Proposition~2]{MR2325344}, Definition~\ref{defn:qf} is equivalent to~\cite[\S3, no.~4, D\'efinition~2]{MR2325344}: a quadratic form on $W$ is a map $Q\colon W\to k$ such that 
	\begin{enumerate}[i)]
  		\item $Q(ax)=a^2 Q(x)$ for all $a\in k$ and $x\in W$, and
  		\item the map $(x,y)\mto Q(x+y)-Q(x)-Q(y)$ is a bilinear form.
	\end{enumerate}
\end{enumerate}

\begin{defn}[Quartic symbol] Let $p\equiv 1\Pmod{4}$ be prime number. If $a\in\Z$ is a quadratic residue (mod $p$) and $(a,p)=1$, we define the \emph{(rational) quartic residue symbol} as
  \[
  \lgd{a}{p}_4 = \pm 1 \equiv a^{\frac{p-1}{4}} \Pmod{p}.
  \]
  For $D=p_1\dotsm p_s$ where $p_i\equiv 1\Pmod{4}$ ($i=1$,\dots,$s$) are distinct primes, and $a\in\Z$ satisfying $(a,D)=1$ and $\lgd{a}{p_i}=1$, we define
  \[
  \lgd{a}{D}_4 = \prod_{i=1}^s \lgd{a}{p_i}_4.
  \]
\end{defn}
Then $\lgd{a}{D}_4=1$ if $a$ is congruent to the fourth power of an integer (mod $D$).

\begin{thm}\label{thm:Hilbert}
  Let $Q_\cB$ the quadratic form on $V_0$ defined by $Q_\cB(D)=\cB(D,D)$ (cf.~\S\ref{subsec:$8$-rank}). Then
  \[
  Q_\cB(D) = \xi\pr{\lgd{\Delta/D}{D}_4}.
  \]
\end{thm}
\paragraph{Remark.}
  This generalizes~\cite[Lemma~7]{MR656862}. Note that unlike in~\cite{MR656862}, we do not assume $\lgd{p_i}{p_j}=1$ for $1\le i\ne j<t$.
\begin{proof} 
  When $D=1$, this is trivial. So we may assume $D\ne 1$.

  If $N\fa=D$, let $(x,y,z)$ be a primitive solution for $x^2=\Delta y^2 + 4D z^2$. Then $z=N\fb$ by Lemma~\ref{lem:Morton}.
We also have $D\mid x$.
Suppose $\Delta=D\cdot D'$, $x=Dx'$. Then
\begin{equation}
  D x'^2 = D' y^2 + 4z^2.
  \label{eqn:z}
\end{equation}
If $p\mid D$, then $p\nmid D'$ and hence $p\nmid z$.
We have
\begin{equation}
\chi_D(\fa)=\lgd{N\fb,\Delta}{D}=\prod_{p\mid D}\lgd{N\fb,\Delta}{p}
= \prod_{p\mid D}\lgd{z}{p} = \lgd{z}{D}.
  \label{eqn:product}
\end{equation}
by Lemma~\ref{lem:Hilbert}. 

Equation (\ref{eqn:z}) implies
\[
  D' y^2 + 4z^2 \equiv 0 \Pmod{p}.
\]
Therefore
\[
\lgd{D'}{p}_4 \lgd{y}{p} = \lgd{z}{p}\lgd{-4}{p}_4.
\]
We are going to compute $\lgd{y}{D}$ and $\lgd{-4}{p}_4$.

Write $y=2^k y'$ where $2\nmid y'$. 
  \[
    \lgd{y'}{D}=\prod_{q\mid y'}\lgd{q}{D}
    =\prod_{q\mid y'}\prod_{p\mid D} \lgd{q}{p} 
    =\prod_{q\mid y'}\prod_{p\mid D} \lgd{p}{q}(-1)^{\frac{p-1}{2}\frac{q-1}{2}}=\prod_{q\mid y'}\lgd{D}{q}.
  \]
Equation (\ref{eqn:z}) implies $D x'^2 \equiv 4z^2 \Pmod{q}$ for all prime factors $q$ of $y'$.
Hence $\lgd{D}{q}=1$ and $\lgd{y'}{D}=1$. 

Now let us consider the possible values of $k$. 
\begin{enumerate}[(i)]
  \item 
    If $k=0$, 
    we have
    $\lgd{y}{D}=\lgd{y'}{D}=1$.
  \item 
    Now assume $k>0$. We have
    $D\equiv 1\Pmod{4}$ and
    $D'\equiv 1\Pmod{4}$.
    From
      $D x'^2 = D' 2^{2k} y'^2 + 4z^2$,
    we see that
    if $k>0$, then $2\mid x'$. Since $\gcd(x,y,z)=1$, we know $2\nmid z$. Hence $z^2\equiv 1\Pmod{8}$.
    Write $y=2\tilde y$ and $x'=2\tilde x$. Then
    \begin{equation}\label{eqn:xyz}
    D\tilde x^2 = D'\tilde y^2 + z^2.
    \end{equation}
    \begin{enumerate}[(a)]
      \item If $k=1$, i.e., $2\nmid\tilde y$, then (\ref{eqn:xyz}) implies
	\[
	\tilde x^2 \equiv 1+1 \Pmod{4},
	\]
	which is impossible.  
      \item If $k=2$, we have 
	$\lgd{y}{D} = \lgd{2}{D}^2 \lgd{y'}{D} = \lgd{y'}{D} =1$.
      \item If $k\ge 3$, then $4\mid\tilde y$. It follows (\ref{eqn:xyz}) that $2\nmid\tilde x$ and hence $\tilde x^2\equiv 1\Pmod{8}$. Then (\ref{eqn:xyz}) implies $D\equiv 1\Pmod{8}$. In other words, $D$ has an even number of prime factors $p$ with $p\equiv 5\Pmod{8}$. Note that $\lgd{2}{p}$ is $1$ if $p\equiv 1\Pmod{8}$ and is $-1$ if $p\equiv 5\Pmod{8}$. Therefore 
	$\lgd{y}{D} = \lgd{y}{D'}\cdot \prod_{p\mid D} \lgd{2}{p} = \lgd{y'}{D} = 1$.
    \end{enumerate}
\end{enumerate}
To summarize, $\lgd{y}{D}=1$ in all cases. 

Note that $\lgd{-4}{p}_4=\lgd{-1}{p}_4\lgd{2}{p}$.  We have $\lgd{-1}{p}_4\equiv (-1)^{\frac{p-1}{4}}\Pmod{p}$ and $\lgd{2}{p}=(-1)^{\frac{p^2-1}{8}}$. No matter whether $p\equiv 1$ or $5\Pmod{8}$, we always have $\lgd{-4}{p}_4=1$.

By (\ref{eqn:product}), we have $\lgd{z}{D}=\lgd{D'}{D}_4$.
\end{proof}

\section{Statement of the main theorem}
\label{sec:thm}
\begin{defn}\label{defn:isotropy}
The \emph{isotropy index} of a quadratic form $Q$ on a vector space $W$ 
is defined to be the maximal dimension of $W'$, where $W'$ is a subspace of $W$ and $Q|_{W'}=0$. We denote the isotropy index of $Q$ by $\rho(Q)$.
\end{defn}
For a quadratic form $Q$, we write
      \[
      \sN(Q)=\Set{\Null(B)}
      {\text{ $B\colon W\times W\to k$ is a bilinear form such that }Q_B=Q}.
      \]
Here 
$\Null(B)=\dim W - \rank(B)$ is the \emph{nullity} of $B$.
\begin{thm}
  \label{thm:main}
  \begin{enumerate}[(I)]
    \item Let $Q_\cB$ be the quadratic form on $V_0$ defined by $Q_\cB(D)=\xi\pr{{\lgd{\Delta/D}{D}_4}}$ as in \S\ref{sec:Q_B}. Let $r_8$ be the $8$-rank of the class group of $F=\Q(\sqrt{\Delta})$. Then 
      \[ 
      r_8\in\sN(Q_\cB).
      \]
    \item 
      Let $Q$ be a quadratic form on an $r$-dimensional vector space over $\F_2$ with isotropy index $\rho$. Then
      \begin{equation*}
      \sN(Q) = S(\rho,r) \df \Set{a}{
      \begin{array}{l}
        0\le a\le\rho,\\
        a\equiv r\Pmod{2}\text{ if }\rho=r, \\
        a=1\text{ if }r=2\text{ and }Q\cong X
      \end{array}
      },
      \end{equation*}
      where $X$ is the quadratic form on $\F_2^2$ defined by $X(x_1e_1+x_2e_2)=x_1x_2$ for the standard basis $\set{e_1,e_2}$ of $\F_2^2$.

      In particular, for the quadratic form $Q_\cB$ in (I), we have $\sN(Q_\cB)=S(\rho,r_4)$, where $r_4=\dim_{\F_2} V_0$ is the $4$-rank of the class group of $F=\Q(\sqrt{\Delta})$.
  \end{enumerate}
\end{thm}
The proof will be given in \S\ref{sec:proof}. 
\begin{cor}\label{cor:r_8}
  $r_8\le\rho(Q_\cB)$.  
\end{cor} 
An immediate consequence of Corollary~\ref{cor:r_8} is
\begin{equation}
  r_8\le\log_2 \#Q_\cB\inv(0).
  \label{eqn:log2}
\end{equation}
For an analogue of \eqref{eqn:log2} for certain real quadratic fields, see Fouvry and Kl\"uners~\cite[Theorem 3 (ii)]{MR2679695}. 

\begin{cor}\label{cor:mod}
  If $Q_\cB=0$, then $r_8\equiv r_4\Pmod{2}$.
\end{cor}

Indeed, Corollaries~\ref{cor:r_8} and \ref{cor:mod} follows from Theorem~\ref{thm:main}~(I) and the easy part of (II), i.e., $\sN(Q)\subset S(\rho,r)$.

\section{Interlude: Quadratic forms over $\F_2$}
In this section we will review the classification of quadratic forms over $\F_2$ which will be used in the proof of Theorem~\ref{thm:main}, and calculate the isotropy index (Definition~\ref{defn:isotropy}) in each case. Some of the material (in more general form) can be found in~\cite[Chapitre 1, \S16]{MR0310083}. 

In this section $Q$ is a quadratic form on a finite-dimensional vector space $W$ over $\F_2$.
It induces an alternating bilinear form $\nabla_Q(x,y) \df Q(x+y) - Q(x) - Q(y)$.

If $W'$ is a subspace of $W$, we define ${W'}^\perp=\Set{x\in W}{\nabla_Q(x,y)=0\text{ for all }y\in W'}$.

$\nabla_Q$ induces a nondegenerate alternating bilinear form $\overline{\nabla}_Q$ on $W/W^\perp$. 
Therefore there exists a basis $\set{e_1,\dots,e_k,e_{k+1},\dots,e_{2k}}$ of $W/W^\perp$ so that $\overline{\nabla}_Q$ can be written as
\[
\begin{pmatrix}
    0 & I \\
    I & 0
\end{pmatrix}
.
\]
The restriction of $Q$ to $W^\perp$ is linear.
\begin{defn}
  The \emph{defect} of $Q$ is defined to be
$d=\dim_{\F_2}Q(W^\perp)$.
\end{defn}
By definition, $d$ equals $0$ or $1$.
\begin{defn}
  The \emph{rank} of $Q$ is defined to be
$\rk(Q)=2k+d$, where $2k=\dim W/W^\perp$ as before.
\end{defn}

It follows that the isotropy index $\rho(Q)=\rho(Q_0)+(n-\rk(Q))$, where
$Q_0$ is the (nondegenerate) quadratic form on $W_0\df W/\ker(Q|_{W^\perp})$ induced from $Q$.

\subsection*{Classification of quadratic forms over $\F_2$}

Let $\set{e_1,\dots,e_n}$ be a basis of $\F_2^n$. 
As a shorthand, we introduce the following quadratic forms:
\begin{itemize}
  \item On $\F_2^1$, let $I(x_1e_1)=x_1^2$.
  \item On $\F_2^2$, let $X(x_1e_1+x_2e_2)=x_1x_2$ and
$Y(x_1e_1+x_2e_2)=x_1^2 + x_1x_2 + x_2^2$. 
  \item On $\F_2^n$, let $O_{n}(\sum_{i=1}^n x_ie_i)=0$.
\end{itemize}
\subsubsection*{Type~1: the rank of $Q$ is odd ($\rk(Q)=2k+1$)}
There exists a basis $\set{e_1, \dots, e_{2k+1}, \dots, e_n}$ such that
\[
    Q\pr{\sum_{i=1}^n x_i e_i}
  = \sum_{i=1}^k x_i x_{k+i} + x_{2k+1}^2
  \cong X^{\op k}\op I\op O_{n-2k-1}.
\]
Then 
\[
\rho(Q_0)=k \quad\text{and}\quad
\rho(Q)=k+(n-\rk(Q)).
\]
\subsubsection*{Type~2: the rank of $Q$ is even ($\rk(Q)=2k$)}
There exists a basis $\set{e_1, \dots, e_{2k}, \dots, e_n}$ such that either
\[
    Q\pr{\sum_{i=1}^n x_i e_i}
 =  \sum_{i=1}^k x_i x_{k+i}
 \cong X^{\op k} \op O_{n-2k}
\quad\text{\textbf{(Type~2.1)}}
\]
or
\[
    Q\pr{\sum_{i=1}^n x_i e_i}
    = \sum_{i=1}^{k-1} x_i x_{k+i}
    + x_k^2 + x_k x_{2k} + x_{2k}^2
    \cong X^{\op(k-1)} \op Y \op O_{n-2k}
\quad\text{\textbf{(Type~2.2)}}.
\]
Then
\[
\rho(Q_0) = \begin{cases}
  k & \text{(Type 2.1)}\\
  k-1 & \text{(Type 2.2)}
\end{cases}
\quad\text{and}\quad
\rho(Q)=\rho(Q_0)+(n-\rk(Q)).
\]

For the proof of the classification (and generalization to finite fields of characteristic $2$), one may refer to~\cite[Chapter VIII, \S199]{MR0104735}.

In both Type 1 and Type 2, we have $2\rho(Q)\ge n-2$.
\subsubsection*{How to determine the type of $Q$?}
There is an easy way to determine the type of $Q$.
It can be checked that the cardinality of the preimage of 0 is
\[
\# Q\inv(0) =
\begin{cases}
  2^{n-1} & \text{(Type~1)},\\
  2^{n-1}+2^{\rho-1} & \text{(Type~2.1)},\\
  2^{n-1}-2^\rho & \text{(Type~2.2)}.
\end{cases}
\]
Therefore, the type of $Q$ can be determined by ``vote'':
\[
Q \text{ is of }
\begin{cases}
  \text{Type~1} & \text{if\quad} \# Q\inv(0) = \# Q\inv(1),\\
  \text{Type~2.1} & \text{if\quad} \# Q\inv(0) > \# Q\inv(1),\\
  \text{Type~2.2} & \text{if\quad} \# Q\inv(0) < \# Q\inv(1).
\end{cases}
\]
The \emph{Arf invariant} of $Q$ is 0 if $Q$ is of Type 2.1 and is 1 if $Q$ is of Type 2.2~\cite{MR0008069}. 

Moreover, the isotropy index $\rho$ of $Q$ can be determined from $\# Q\inv(0)$ if $Q$ is of Type~2.

\section{Proof of the main theorem}
\label{sec:proof}
  For (I): from \S\ref{subsec:$8$-rank} and the definition of $Q$ we see $r_8\in\sN(Q_\cB)$.

  For (II): We first show $\sN(Q)\subset S(\rho,r)$.
  Suppose $B$ is a bilinear form which induces the quadratic form $Q_B=Q$. Then $\Null(B)\le\rho(Q)$ by definition.

  If $\rho=r$, i.e., $Q=0$, then $B$ is an alternating bilinear form. Therefore $\rank(B)\equiv 0\Pmod{2}$ and $\Null(B)\equiv r\Pmod{2}$.

  If $r=2$ and $Q\cong X$, we have $\Null(B)=1$.

  Now we show $\sN\supset S(\rho,r)$.
  Suppose $a\in S(\rho,r)$. We are going to show that there is a bilinear form $B$ such that $Q_B=Q$ and $\Null(B)=a$.

  We will adopt the following notation:  if $S_1$, $S_2\subset\Z$, we write
  \[
    S_1+S_2=\Set{s_1+s_2}{s_1\in S_1,s_2\in S_2}.
  \]

  We define
  \[
  \sB(Q)=\Set{B}
  {\text{ $B$ is a bilinear form such that }Q_B=Q}.
  \]
  \begin{lem}
    Let $Q_1$ and $Q_2$ be two quadratic forms. Then
    $\sN(Q_1)+\sN(Q_2)\subset\sN(Q_1\op Q_2)$.
  \end{lem}
  \begin{proof}
    Obvious.
  \end{proof}

  According to the classification of quadratic forms over $\F_2$, it suffices to prove that $\sN\supset S(\rho,r)$ for the following cases:
  \begin{enumerate}[(I)]
    \item $X^{\op k}\op O_m$ for $k\ge 0$, $m\ge 0$.
    \item $X^{\op k} \op O_m\op Y$ for $k\ge 0$, $m\ge 0$.
    \item $X^{\op k} \op O_m\op I$ for $k\ge 0$, $m\ge 0$.
  \end{enumerate}
  In all cases, the isotropy index $\rho$ is equal to $k+m$.

Now we will check the theorem by exhausting all the cases.
  \begin{enumerate}[(i)]
    \item $\rho(O_1)=1$. We have $\sN(O_1)=\set{1}$.
    \item $\rho(O_2)=2$. Since $\sB(O_2)=\set{\mat{0}{0}{0}{0},\mat{0}{1}{1}{0}}$, we get  $\sN(O_2)=\set{0,2}$.
    \item We will show that $\sN({O_m})=\Set{a}{0\le a\le m,\,a\equiv m\Pmod{2}}$ for $m>2$ by induction. In fact,
      \begin{eqnarray*}
        \sN(O_m) &\supset& \sN(O_{m-2})+\sN(O_2) \\
	&=& \Set{a}{0\le a\le m-2,\,a\equiv m\Pmod{2}}+\set{0,2}\\
	&=& \Set{a}{0\le a\le m,\,a\equiv m\Pmod{2}}.
      \end{eqnarray*}
    \item $\rho(X)=1$. We have $\sB(X)=\set{\mat{0}{1}{0}{0},\mat{0}{0}{1}{0}}$. Therefore $\sN(X)=\set{1}$.
    \item $\rho(X\op O_1)=2$. We have
      \[
        \sN(X\op O_1)\supset\sN(X)+\sN(O_1)=\set{2}.
      \]
      Since
      \[
      \begin{pmatrix}
	0 & 1 & 1\\
	0 & 0 & 1\\
	1 & 1 & 0
      \end{pmatrix},\;
      \begin{pmatrix}
	0 & 1 & 1\\
	0 & 0 & 0\\
	1 & 0 & 0
      \end{pmatrix}
      \in\sB(X\op O_1)
      \]
      and have nullity 0 and 1 respectively,
      we have $\sN(X\op O_1)=\set{0,1,2}$.
    \item
      $\rho(X\op O_2)=3$. On the other hand,
      \begin{eqnarray*}
       \sN(X\op O_2)
      &=&  \sN((X\op O_1)\op O_1)\\
      &\supset& \sN(X\op O_1)+\sN(O_1) \\
      &=& \set{0,1,2}+\set{1}
      =\set{1,2,3}.
      \end{eqnarray*}
      It can be checked that
      \[
      \begin{pmatrix}
	0 & 1 & 1 & 1\\
	0 & 0 & 0 & 1\\
	1 & 0 & 0 & 0\\
	1 & 1 & 0 & 0
      \end{pmatrix}
      \in\sB({X\op O_2})
      \]
      and has nullity 0.
      Therefore, $\sN(X\op O_2)=\set{0,1,2,3}$.
    \item We will show that $\sN(X\op O_m)=\set{0,1,2,\dots,m+1}$ for $m\ge 3$ by induction. In fact,
      \begin{eqnarray*}
        \sN(X\op O_m)
        &=& \sN((X\op O_{m-2})\op O_2) \\
        &\supset& \sN(X\op O_{m-2}) + \sN(O_2)\\
	&=& \set{0,1,\dots,m-1} + \set{0,2}\\
	&=& \set{0,1,\dots,m+1}.
      \end{eqnarray*}
    \item
      $\rho(X\op X)=2$. We have
      $\sN(X\op X)\supset\sN(X)+\sN(X)=\set{2}$.
      Since
      \[
      \begin{pmatrix}
	0 & 1 & 1 & 1\\
	0 & 0 & 0 & 1\\
	1 & 0 & 0 & 1\\
	1 & 1 & 0 & 0
      \end{pmatrix},\;
      \begin{pmatrix}
	0 & 1 & 0 & 1\\
	0 & 0 & 0 & 0\\
	0 & 0 & 0 & 1\\
	1 & 0 & 0 & 0
      \end{pmatrix}\in\sB(X\op X)
      \]
      and have nullity 0 and 1 respectively, we have
      $\sN(X,X)=\set{0,1,2}$.
    \item
      $\rho(X\op X\op O_1)=3$. On the other hand,
      $\sN(X\op X\op O_1)\supset\sN(X\op X)+\sN(O_1)=\set{0,1,2}+\set{1}=\set{1,2,3}$. Note that
      \[
      \begin{pmatrix}
	0 & 1 & 0 & 1 & 1\\
	0 & 0 & 0 & 0 & 1\\
	0 & 0 & 0 & 1 & 0\\
	1 & 0 & 0 & 0 & 1\\
	1 & 1 & 0 & 1 & 0
      \end{pmatrix}\in\sB(X\op X\op O_1)
      \]
      and has nullity 0. It follows that
      $\sN(X\op X\op O_1)=\set{0,1,2,3}$.
    \item We will show that the theorem holds for all $X\op X\op O_m$ for $m\ge 2$ by induction. In fact, we have
      $\sN(X\op X\op O_m)\supset\sN(X\op X\op O_{m-2})+\sN(O_2)=\set{0,1,2,\dots,m}+\set{0,2}=\set{0,1,2,\dots,m+2}$ for $m>2$.
    \item
      $\rho(X\op X\op X)=3$. We have
      $\sN(X\op X\op X)\supset\sN(X\op X)+\sN(X)=\set{0,1,2}+\set{1}=\set{1,2,3}$. Besides,
      \[
      \begin{pmatrix}
	0 & 1 & 1 & 0 & 0 & 0\\
	0 & 0 & 1 & 1 & 0 & 0\\
	1 & 1 & 0 & 1 & 0 & 1\\
	0 & 1 & 0 & 0 & 1 & 0\\
	0 & 0 & 0 & 1 & 0 & 1\\
	0 & 0 & 1 & 0 & 0 & 0
      \end{pmatrix}\in\sB(X\op X\op X)
      \]
      and has nullity 0. Therefore
      $\sN(X\op X\op X)=\set{0,1,2,3}$.
    \item The theorem holds for all $X^{\op k}\op O_m$ by induction for $k\ge 3$ and $k+m\ge 4$, because
      \begin{eqnarray*}
	\sN(X^{\op k}\op O_m)
	&\supset& \sN(X^{\op (k-2)} \op O_m) + \sN(X\op X)\\
      &=&  \set{0,1,2,\dots,k-2+m} + \set{0,1,2}\\
      &=&  \set{0,1,2,\dots,k+m}.
      \end{eqnarray*}
    \item $\rho(Y)=0$. On the other hand, $\sB(Y)=\set{\mat{1}{1}{0}{1},\mat{1}{0}{1}{1}}$ and $\sN(Y)=\set{0}$.
    \item
      $\rho(Y\op O_1)=1$. Since
      \[
      \begin{pmatrix}
	1 & 1 & 0\\
	0 & 1 & 0\\
	0 & 0 & 0
      \end{pmatrix},\;
      \begin{pmatrix}
	1 & 1 & 1\\
	0 & 1 & 0\\
	1 & 0 & 0
      \end{pmatrix}
      \in\sB(Y\op O_1)
      \]
      and have nullity 1 and 0 respectively, we get
      $\sN({Y\op O_1})=\set{0,1}$.
    \item
      $\rho(Y\op O_2)=2$. On the other hand,
      $\sN(Y\op O_2)\supset\sN(Y)+\sN(O_2)=\set{0}+\set{0,2}=\set{0,2}$ and
      $\sN(Y\op O_2)\supset\sN(Y\op O_1)+\sN(O_1)=\set{0,1}+\set{1}=\set{1,2}$.
      Therefore, $\sN(Y\op O_2)=\set{0,1,2}$.
    \item We prove for all $Y\op O_m$ ($m\ge 3$) by induction:
      \begin{eqnarray*}
      \sN(Y\op O_m) &\supset& \sN(Y\op O_{m-2})+\sN(O_2)\\
      &=& \set{0,1,\dots,m-2}+\set{0,2}=\set{0,1,\dots,m}.
      \end{eqnarray*}
    \item
      $\rho(X\op Y)=1$. Since
      \[
      \begin{pmatrix}
	0 & 1 & 0 & 0\\
	0 & 0 & 0 & 0\\
	0 & 0 & 1 & 1\\
	0 & 0 & 0 & 1
      \end{pmatrix},\;
      \begin{pmatrix}
	0 & 1 & 1 & 1\\
	0 & 0 & 1 & 0\\
	1 & 1 & 1 & 1\\
	1 & 0 & 0 & 1
      \end{pmatrix}
      \in\sB(X\op Y)
      \]
      and have nullity 0 and 1 respectively, we know that
      $\sN(X\op Y)=\set{0,1}$.
    \item $\rho(X^{\op k}\op Y\op O_m)=k+m$. For $k\ge 1$ and $k+m\ge 2$, we have 
      $\sN(X^{\op k}\op Y\op O_m)\supset\sN(X^{\op k}\op O_m)+\sN(Y)=\set{0,1,\dots,k+m}$.
    \item $\rho(I)=0$. We also have $\sB(I)=\set{(1)}$ and hence $\sN(I)=\set{0}$.
    \item $\rho(I\op O_1)=1$. In this case, $r_4=2$. We see that
      \[
      \sB(I\op O_1) = \set{\mat{1}{0}{0}{0},\mat{1}{1}{1}{0}}.
      \]
      Therefore, $\sN(I\op O_1)=\set{0,1}$.
    \item $\rho(I\op O_2)=2$. We have $\sN(I\op O_2)\supset\sN(I)+\sN(O_2)=\set{0}+\set{0,2}$ and $\sN(I\op O_2)\supset\sN(I\op O_1)+\sN(O_1)=\set{0,1}+\set{1}=\set{1,2}$.
      Therefore $\sN(I\op O_2)=\set{0,1,2}$.
    \item $\rho(I\op O_m)=m$. We will show that $\sN(I\op O_m)=\set{0,1,\dots,m}$ for $m>2$ by induction. In fact,
      \begin{eqnarray*}
	\sN(I\op O_m) &\supset& \sN(I\op O_{m-2})+\sN(O_2)\\
	&=& \set{0,1,\dots,m-2}+\set{0,2} = \set{0,\dots,m}.
      \end{eqnarray*}
    \item $\rho(X\op I)=1$. We also have $\sN(X\op I)\supset\sN(X)+\sN(I)=\set{1}$. Since
      \[
       \begin{pmatrix}
	 0 & 1 & 1\\
	 0 & 0 & 1\\
	 1 & 1 & 1
       \end{pmatrix}\in\sB(X\op I)
      \]
      and has nullity $0$, we get $\sN(X\op I)=\set{0,1}$.
    \item $\rho(X^{\op k}\op I\op O_m)=k+m$. For $k\ge 1$ and $k+m\ge 2$, we have
      $\sN(X^{\op k}\op I\op O_m)\supset\sN(X^{\op k}\op O_m)+\sN(I)=\set{0,1,\dots,k+m}$.
  \end{enumerate}
  All cases exhausted, we conclude that Theorem~\ref{thm:main} is true.

\paragraph{Remark.} 
Theorem~\ref{thm:main} states that the $8$-rank $r_8\in \sN(Q_\cB)$. 
If $r_4=1$, we have $Q_\cB\cong O_1$ or $I$. Since $\sN(O_1)=\set{1}$ and $\sN(I)=\set{0}$, $r_8$ is determined by $Q_\cB$. However, in general $r_8$ may not be determined by $Q_\cB$. 
For example, if $r_4=2$, $Q_\cB\cong X$ or $Y$ or $O_2$ or $I\op O_1$. We know from Theorem~\ref{thm:main} that $\sN(X)=\set{1}$, $\sN(Y)=\set{0}$, $\sN(O_2)=\set{0,2}$ and $\sN(I\op O_1)=\set{0,1}$. In the latter two cases we have $\#\sN(Q_\cB)>1$. 
In all cases,  all values in $\sN(Q_\cB)$ actually appear as the $8$-ranks of the class groups of infinitely many imaginary quadratic number fields (cf.~the proof of Theorem~2 in~\cite{MR656862}).

\addcontentsline{toc}{section}{Bibliography}
\bibliography{8-rank}
\bibliographystyle{abbrv}
\end{document}